\documentclass{IEEEtran}

\IEEEoverridecommandlockouts                              

\usepackage{enumerate}
\usepackage{amsmath}
\usepackage[dvipsnames]{xcolor}
\usepackage{amssymb}
\usepackage{amsfonts}
\usepackage{graphicx}
\usepackage{epstopdf}
\usepackage{amsthm}
\usepackage{amsmath}
\usepackage{cancel}
\usepackage{mathrsfs}
\usepackage{mathdots}
\usepackage{euscript}
\usepackage{amscd}
\usepackage{cite}
\usepackage{placeins}
\usepackage{algorithm2e}
\usepackage{tikz}
\usetikzlibrary{snakes,arrows,shapes}

\graphicspath{{images/}}

\newtheorem{thm}{Theorem}
\newtheorem{cor}{Corollary}
\newtheorem{lemma}{Lemma}
\newtheorem{prop}{Proposition}

\theoremstyle{definition}\newtheorem{prob}{Problem}
\theoremstyle{plain}
\theoremstyle{remark}

\newcommand{\bmat}{\left[ \begin{matrix}}
	\newcommand{\emat}{\end{matrix} \right]}

\newcommand{\trace} {\mbox{\rm tr}}

\newcommand{\E}{{\mathbb E}}
\newcommand{\Rbb}{\mathbb R}

\newcommand{\Dbb}{\mathbb D}
\newcommand{\Cbb}{\mathbb C}

\newcommand{\Zbb}{\mathbb Z}
\newcommand{\Tbb}{\mathbb T}

\newcommand{\range}{\mathrm{Range}}

\newcommand{\Hfrak}{\mathfrak{H}}
\newcommand{\Sscr}{\mathscr{S}}
\newcommand{\Lscr}{\mathscr{L}}

\begin{document}
\title{On the Existence of a Solution to a Spectral Estimation Problem \emph{\`a la} Byrnes-Georgiou-Lindquist
\thanks{The first author was supported by the China Scholarship Council (CSC) under File No.~201506230140.} 
}
\author{
Bin Zhu and Giacomo Baggio
\thanks{B. Zhu and G. Baggio are with the
Department of Information Engineering, University of Padova, via
Gradenigo 6/B, 35131 Padova, Italy; e-mail:  
{\tt zhubin@dei.unipd.it, baggio@dei.unipd.it}
}}
\markboth{}
{Constrained Spectral Estimation}
\maketitle

\begin{abstract}
A parametric spectral estimation problem in the style of Byrnes, Georgiou, and Lindquist was posed in \cite{FPZ-10}, but the existence of a solution was only proved in a special case. Based on their results, we show that a solution indeed exists given an arbitrary matrix-valued prior density. The main tool in our proof is the topological degree theory.
\end{abstract}

 
 
\section{Introduction}


This note concerns a spectral estimation problem subjected to a generalized moment constraint. {The setup of the problem (scalar version) was first introduced by Byrnes, Georgiou, and Lindquist in \cite{BGL-THREE-00}, and then further elaborated in \cite{Georgiou-L-03} in order to allow for an \emph{a priori} information. This formulation, known under the name of \emph{THREE-like spectral estimation}, has now become nearly standard and includes as special cases some important problems in the field of systems and control such as \emph{covariance extension} (cf. e.g., \cite{Kalman,Gthesis,Georgiou-87,BLGM-95,BGL-98,SIGEST-01}) and \emph{Nevanlinna--Pick interpolation} (cf. \cite{Georgiou-87-NP,BGL-01,BLN-03} and references therein).}

It is worth pointing out that moment problems \cite{KreinNudelman,Grenander_Szego} form a special class of inverse problems that are typically \emph{not} well-posed in the sense of Hadamard. To remedy this, the mainstream approach today is to first define an {entropy-like} distance index $d(\Phi,\Psi)$ between two bounded and coercive spectral densities, and then to find the ``best'' $\Phi$ given the prior $\Psi$ by minimizing the distance index subjected to the generalized moment constraint. Still, it is not trivial to solve such an optimization problem. {Indeed, although the dual problem is typically convex, the dual variable (i.e., the Lagrange multiplier) is a Hermitian matrix that lives in an \emph{open, unbounded} domain and this usually gives rise to a number of numerical issues. With reference to the scalar case, results in this direction include the aforecited \cite{Georgiou-L-03}, in which the chosen distance index is the Kullback--Leibler divergence (cf. also \cite{PavonF-06,FPR-07,FRT-11,B17}), and \cite{Z14rat}, where a general family of divergences (the Alpha divergence family) is considered. In the multivariate case, the problem becomes much more challenging and its feasibility strongly depends on the selected distance. We mention, in particular, the papers \cite{Georgiou-06}, where a multivariate extension of the Kullback--Leibler divergence, the quantum relative entropy, is considered; \cite{FPR-08,RFP-09}, which deal with a sensible generalization of the Hellinger distance; and \cite{FMP-12,GL-17}, where the selected distance index coincides with the multivariate Itakura--Saito distance. It is worth remarking that the latter two approaches lead to rational solutions with bounded McMillan degrees when the prior is rational. Finally, \cite{Z14} and \cite{Z15} introduce two more general frameworks based on the notion of Beta and Tau divergence families, wherein the multivariate Kullback--Leibler and Itakura--Saito distance can be recovered as particular cases.}


There are a few attempts in directions different from optimization; see \cite{FPZ-10,PZ-converge-17,Picci-Z-017}. In particular, a parametric family of spectral densities was introduced in \cite{FPZ-10}, and a certain map from the parameter space to the space of (generalized) moments was studied in the light of a Hadamard-type global inverse function theorem \cite{BL_global_inverse-07}. {The proposed parametrization has been shown to be amenable for the implementation of a matricial version of an extremely simple and efficient fixed-point algorithm introduced in \cite{PavonF-06}, whose convergence properties have been investigated in \cite{FPR-07,FRT-11,B17}.} However, the result in \cite{FPZ-10} was not satisfactory because the authors only showed that a solution exists when the prior $\Psi$ has a very special structure. In fact, this is the motivation of the current note. As a continuation of the work in \cite{FPZ-10}, here we will show that a solution to the parametric spectral estimation problem exists given \emph{any} fixed matrix-valued prior density that is bounded and coercive. {Of course the problem is still open to a large extent since uniqueness of the solution (and, {\em a fortiori}, well-posedness of the problem) is not known, and the convergence properties of the algorithm proposed in \cite{FPZ-10} to compute a solution are yet to be examined.}

The main machinery behind our existence proof is the \emph{topological degree theory} from nonlinear analysis. As a historical remark, Georgiou was the first to apply the degree theory to rational covariance extension \cite{Georgiou-87,Georgiou-87-NP,Gthesis} to show existence of a solution, and it was further developed by Byrnes, Lindquist, and coworkers \cite{BLGM-95} to prove the uniqueness and well-posedness. These theories were established before the discovery of the cost function in the optimization framework \cite{BGL-01,BGL-98,SIGEST-01}, which was later called generalized entropy criterion.


The outline of this note is as follows.
In Section \ref{SecProb}, we first set up some notations before reviewing the problem formulation. The important special case of multivariate covariance extension is detailed for illustration.
Our main result is presented in Section~\ref{SecExist}. A part of the degree theory is reviewed in order to carry out our proof.
{We conclude with some open questions on the uniqueness of the solution 
and convergence of an algorithm leading to such a solution.}

\subsection*{List of symbols}

\begin{itemize}
    	\item $\E$, mathematical expectation.
    	\item $\Zbb$, the set of integers.
    	\item $\Cbb$, the complex {plane}.
    	\item $\Dbb$, the open complex unit disk $\{z\in\Cbb:\,|z|<1\}$.
    	\item $\Tbb\equiv\partial\Dbb$, the unit circle, where $\partial$ stands for the boundary.
    	\item {$\mathrm{GL}(n,\Cbb)$, group of $n\times n$ invertible complex matrices}.
	\item $\Hfrak_{n}$, the vector space of $n\times n$ Hermitian matrices.
	\item $\Hfrak_{+,n}$, the subset of $\Hfrak_{n}$ that contains positive definite matrices.
	\item $C(\Tbb;\Hfrak_m)$, the space of $\Hfrak_{m}$-valued continuous functions on $\Tbb$.
	\item $\Sscr_m$, the family of $\Hfrak_{+,m}$-valued functions defined on $\Tbb$ that are bounded and coercive.
    	\item {$(\cdot)^{*}$, complex conjugate transpose. When considering a rational matrix-valued function $G(z)$,  $G^{*}(z)$ stands for the analytic continuation of the function that for $z\in\Tbb$ equals the complex conjugate transpose of $G(z)$.}
	\item {$(\cdot)^{-*}$, shorthand for $[(\cdot)^{-1}]^{*}$.}
\end{itemize}

\section{A multivariate spectral estimation problem}\label{SecProb}



Consider a linear system with a state-space representation
\begin{equation}\label{filter_bank}
x(t+1)=Ax(t)+By(t),
\end{equation}
where $A\in\Cbb^{n\times n}$ is Schur stable, i.e., has all its eigenvalues in $\Dbb$, $B\in\Cbb^{n\times m}$ is of full column rank ($n\geq m$). Moreover, the pair $(A,B)$ is assumed to be \emph{reachable}. The input process $y(t)$ is zero-mean wide-sense stationary with an unknown spectral density matrix $\Phi(z)$. The transfer function of (\ref{filter_bank}) is just
\begin{equation}\label{trans_func}
G(z)=(zI-A)^{-1}B,
\end{equation}
which can be interpreted as a bank of filters. {An estimate of the steady-state covariance matrix $\Sigma:=\E\{x(t)x(t)^*\}$ of the state vector $x(t)$ is assumed to be known. (For the problem of estimating covariance matrices in this setting, we refer to \cite{Zorzi-F-12, FPZ-12,ning2013geometry}).} Hence we have
\begin{equation}\label{mmt_constraint}
\int G\Phi G^*=\Sigma,
\end{equation}
where the function is integrated on $\Tbb$ with respect to the normalized Lebesgue measure $\frac{\mathrm{d}\theta}{2\pi}$. This notation will be adopted throughout the note.

Given the matrix $\Sigma\in\Hfrak_{+,n}$, we want to estimate the spectral density $\Phi$ such that the generalized moment constraint (\ref{mmt_constraint}) is satisfied. For example, consider the following choice of the matrix pair $(A,B)$:
\begin{equation}
A=\bmat 0&I_m&0&\cdots&0 \\ 
0&0&I_m&\cdots&0 \\
\vdots&\vdots& &\ddots&\vdots \\
0&0&0&\cdots&I_m \\
0&0&0&\cdots&0\emat, \quad
B=\bmat 0 \\ 0 \\ \vdots \\ 0 \\ I_m\emat.
\end{equation}
Here each block in $A$ or $B$ is of $m\times m$ and $A$ is a $(p+1)\times(p+1)$ block matrix while $B$ is a $(p+1)$-block column vector. It is easy to verify that in this case
\begin{equation}\label{trans_func_covext}
G(z)=(zI-A)^{-1}B=\bmat z^{-p-1}I_m \\ z^{-p}I_m \\ \vdots \\z^{-1}I_m \emat,
\end{equation}
Symbolically, the steady state vector
\begin{equation}
x(t)=G(z)y(t)=\bmat y(t-p-1) \\ \vdots \\ y(t-2) \\ y(t-1) \emat,
\end{equation}
and the covariance matrix $\Sigma$ has a block-Toeplitz structure, i.e.,
\begin{equation}
\Sigma=\bmat \Sigma_0&\Sigma_1^*&\Sigma_2^*&\cdots&\Sigma_p^* \\
\Sigma_1&\Sigma_0&\Sigma_1^*&\cdots&\Sigma_{p-1}^* \\
\Sigma_2&\Sigma_1&\Sigma_0&\cdots&\Sigma_{p-2}^* \\
\vdots&\vdots&\ddots&\ddots&\vdots \\
\Sigma_p&\Sigma_{p-1}&\cdots&\Sigma_1&\Sigma_0\emat,
\end{equation}
where $\Sigma_k:=\E\{y(t+k)y(t)^*\}\in\Cbb^{m\times m}$ with a slight abuse of notation.
In fact, the constraint (\ref{mmt_constraint}) is equivalent to the set of \emph{moment equations}
\begin{equation}\label{mmt_eqns}
\int_{-\pi}^{\pi}e^{jk\theta}\Phi(e^{j\theta})\frac{d\theta}{2\pi}=\Sigma_k,\quad k=0,1,\dots,p.
\end{equation}
To find a spectral density $\Phi$ satisfying (\ref{mmt_eqns}) is the classical \emph{covariance extension problem} \cite{Grenander_Szego}. 

In general, existence of $\Phi\in\Sscr_m$ satisfying (\ref{mmt_constraint}) is not trivial. Such feasibility problem was addressed in \cite{Georgiou-02,georgiou2002spectral}, see also \cite{FPZ-10,FPZ-12,Zorzi-F-12,FPR-07,FPR-08,RFP-09,RFP-10-well-posedness,FMP-12}. In order for $\Sigma>0$ to be a state covariance, a certain Lyapunov-like equation has to be solvable or an equivalent rank condition must hold. Interested readers can consult the references for details. Here we shall take the feasibility as a standing assumption. More precisely, let us define the linear operator $\Gamma\colon C(\Tbb;\Hfrak_m)\to\Hfrak_n$ as
\begin{equation}
\Gamma\colon\Phi\mapsto\int G\Phi G^*.
\end{equation}
Then we assume that the covariance matrix $\Sigma\in\range\,\Gamma$. Various properties of the set $\range\,\Gamma$ are elaborated in e.g., \cite[Sec.~III]{FPZ-12}. In particular, by Proposition 3.1 of that paper, $\range\,\Gamma\subset\Hfrak_n$ is a linear space with real dimension $m(2n-m)$. 



Moreover, define the set
\begin{equation}
\Lscr_+:=\{\Lambda\in\Hfrak_{n}\;:\;G^*(z)\Lambda G(z)>0,\ \forall z\in\Tbb\}.
\end{equation}
By the continuous dependence of eigenvalues on the matrix entries, one can verify that $\Lscr_+$ is an open subset of $\Hfrak_{n}$. For $\Lambda\in\Lscr_+$, take $W_\Lambda$ as the unique stable and minimum phase {(right)} spectral factor of $G^*\Lambda G$ \cite[Lemma 11.4.1]{FPZ-10}, i.e.,
\begin{equation}\label{W_Lambda}
G^*\Lambda G=W_\Lambda^*W_\Lambda.
\end{equation}
Our problem is formulated as follows.

\begin{prob}\label{spec_estimation}
	Given the filter bank $G(z)$ in (\ref{trans_func}), let $\Sigma\in\range_+\Gamma:=\range\,\Gamma\,\cap\Hfrak_{+,n}$ and $\Psi\in\Sscr_m$. Find $\Lambda\in\Lscr_+$ such that 
	\begin{equation}\label{Phi_Gamma}
		\Phi_\Lambda:=W_\Lambda^{-1}\Psi W_\Lambda^{-*}
	\end{equation}
	satisfies
	\begin{equation}\label{constraint}
		\int G\Phi_\Lambda G^*=\Sigma.
	\end{equation}
\end{prob}

Define $\Lscr_+^\Gamma:=\Lscr_+\cap\range\,\Gamma$, and consider the map $\omega\colon\Lscr_+^\Gamma\to\range_+\Gamma$ given by
\begin{equation}\label{map_omega}
\omega\colon\Lambda\mapsto\int G\Phi_\Lambda G^*.
\end{equation}
As indicated in \cite{FPZ-10} and will be clear in Subsection \ref{subsec:proof}, this is a continuous map between open subsets of the linear space $\range\,\Gamma$, and Problem \ref{spec_estimation} is feasible if and only if $\omega$ is surjective. Theorem 11.4.3 in \cite{FPZ-10} guarantees such surjectivity when the prior is a scalar density times a positive definite matrix. In the next section, we shall extend that result to accommodate an arbitrary matrix spectral density $\Psi$.

\section{Existence of a solution}\label{SecExist}


The proof of our main result relies on the notion of topological degree of a continuous map. The degree theory forms an important part of differential topology and is closely related to fixed-point theory, cf. \cite[Ch. I]{OR_degree} for a rather informative historical account. In particular, the degree theory is a powerful tool to prove existence of a solution to a system of nonlinear equations. There are several versions of the theory for different types of maps. Although the maps that we consider in this note are between open subsets of the Euclidean space, we shall use the more general degree theory for continuous maps between smooth, connected, boundary-less manifolds. Some main points of the theory are reviewed below.

\subsection{A short review of the degree theory}
We mainly follow the lines of \cite[Ch.~III]{OR_degree}. Suppose $U,\,V\subset\Rbb^n$ are open and connected, and $f\colon U\to V$ is a {\em proper} $C^1$ function. Recall that $f$ is called proper if the preimage of every compact set in $V$ is compact in $U$. Our major concern is solvability of the equation
\begin{equation}
	f(x)=y.
\end{equation}
A point $y\in V$ is called a regular value of $f$ if either
\begin{enumerate}[(i)]
	\item for any $x\in f^{-1}(y)$, $\det f'(x)\neq0$  or
	\item $f^{-1}(y)$ is empty.
\end{enumerate}
Here $f^{-1}(y)$ denotes the preimage of $y$ under $f$, i.e., the set
\[\{x\in U\,:\,f(x)=y\},\]
and $f'(x)$ denotes the Jacobian matrix of $f$ evaluated at $x$.
Let $y$ be a regular value of type (i), and the degree of $f$ at $y$ is defined as
\begin{equation}\label{def_degree}
	\deg(f,y):=\sum_{f(x)=y}\mathrm{sign}\,\det f'(x),
\end{equation}
where the sign function
\[\mathrm{sign}(x)= \left\{ \begin{array}{ll}
1 & \textrm{if $x>0$} \\
-1 & \textrm{if $x<0$} \\
\end{array} \right.\]
and not defined at $0$.

Throughout this note, properness will be a crucial property of our function. Since $f$ is proper, one can show that the preimage $f^{-1}(y)$ is finite following the classical inverse function theorem, and hence the sum above is well defined. For regular values of type (ii), we set $\deg(f,y)=0$. Moreover, the set of regular values is dense in $V$ by Sard--Brown Theorem \cite[p.~63]{OR_degree}. 
Further properties of the degree related to our problem are listed below:
\begin{itemize}
	\item The degree of $f$ at $y$ does not depend on the choice of regular value. Therefore, we can define the degree of $f$ as
	\[\deg(f)=\deg(f,y)\]
	for any regular value $y$.
	\item If $\deg(f)\neq0$, then for any $y\in V$, there exists $x\in U$ such that $f(x)=y$, that is, the map $f$ is surjective. A proof of this fact can be found in \cite[p.~1849]{BLGM-95}.
	\item Homotopy invariance. If $H\colon U\times [0,1]\to V,\ (x,t)\mapsto y$ is jointly continuous in $(x,t)$ and proper, then $\deg(H_t,y)$ is defined and independent of
	$t\in[0,1]$. Here $H_t\colon U\to V$ is defined by $H_t(x)=H(x,t)$.
\end{itemize}

One important point of theory is that degree can be defined for continuous functions through approximation by smooth functions\cite[Proposition and Definition~3.1, p.~111]{OR_degree}, and (\ref{def_degree}) is just a way of computing it in the special case of $C^1$ \cite[Remark p.~71]{JTSchwartz}. In particular, the homotopy invariance of the degree holds in the continuous case \cite[Proposition~3.4, p.~112]{OR_degree}.


\subsection{Proof of existence}\label{subsec:proof}

Our main theorem will be preceded by some lemmas. Take $\Psi=I$ the identity matrix, and the map $\omega$ would reduce to
\begin{equation}\label{omega_tilde}
	\begin{split}
		\tilde{\omega}\colon \Lscr_+^\Gamma &\to\range_+\Gamma \\
		\Lambda & \mapsto\int G(G^*\Lambda G)^{-1} G^*.
	\end{split}
\end{equation}

\begin{lemma}\label{lem_smooth_omega_tilde} 
	The map $\tilde{\omega}$ is continuously differentiable.
\end{lemma}
\begin{proof}
	The map
	\begin{equation}\label{mat_inverse}
	\mathrm{GL}(n,\Cbb)\to\mathrm{GL}(n,\Cbb)\, :\, X\mapsto X^{-1}
	\end{equation}
	is smooth, which follows from Cramer's rule in linear algebra. Hence, the function $\tilde{F}_\Lambda(e^{j\theta}):=G(G^*\Lambda G)^{-1} G^*$ inside the integral of (\ref{omega_tilde}) is also smooth in $\Lambda$. Moreover, since $G$ is a rational function, all the partial derivatives of $\tilde{F}_\Lambda(e^{j\theta})$ with respect to $\Lambda$ are continuous in $\theta$ (and $\Lambda$). Then by Leibniz's rule for differentiation under the integral sign, partial derivatives of $\tilde{\omega}$ of all orders exist.

	Next, we show that the first order partial derivatives are continuous. For the time being, let us consider the map $\tilde{\omega}$ defined on $\Lscr_+$. (We made the domain restricted to the intersection with $\range\,\Gamma$ out of the consideration of dimensionality.) From \cite{brookes2005matrix}, the differential of the map (\ref{mat_inverse}) at $X$ is given by
	\begin{equation*}
	\Cbb^{n\times n}\to\Cbb^{n\times n}\,:\, V\mapsto-X^{-1}VX^{-1}.
	\end{equation*}
	Using this fact, the differential of $\tilde{\omega}$ at $\Lambda\in\Lscr_+$ is
	\begin{equation}\label{diff_omega_tilde}
	\delta\Lambda\mapsto\delta\tilde{\omega}_\Lambda=-\int G(G^*\Lambda G)^{-1}(G^*\delta\Lambda G)(G^*\Lambda G)^{-1}G^*
	\end{equation}
	such that $\delta\Lambda\in\Hfrak_n$ and $\Lambda+\delta\Lambda$ stays in $\Lscr_+$. Let us denote the integrand in (\ref{diff_omega_tilde}) by $\delta\tilde{F}_{\Lambda,\delta\Lambda}(e^{j\theta})$. For a fixed $\delta\Lambda$, one can see that the differential $\delta\tilde{\omega}_\Lambda(\delta\Lambda)$ is continuous w.r.t. $\Lambda$. To see this fact, let a sequence $\{\Lambda_k\}_{k\geq1}\subset\Lscr_+$ converge to some $\bar{\Lambda}\in\Lscr_+$ as $k\to\infty$. Notice that, the corresponding sequence of matrix-valued functions $\{G^{*}\Lambda_{k}G\}_{k\ge 1}$ is such that $G^{*}(e^{j\theta})\Lambda_{k}G(e^{j\theta})>0$ for all $\theta\in[-\pi,\pi]$ and $k$. Since the eigenvalues of a continuous matrix-valued function $F\colon [a,b]\to \Cbb^{n\times n}, \, \theta\mapsto F(\theta)$, depend continuously on $\theta$ \cite[Cor.~VI.1.6]{Bhatia}, we have that $G^{*}\Lambda_{k}G \ge \mu_{k} I$ where \[\mu_{k}:=\min_{\theta}\lambda_{\min}\left(G^{*}(e^{j\theta})\Lambda_{k}G(e^{j\theta})\right)>0\] 
	and $\lambda_{\min}(\cdot)$ denotes the smallest eigenvalue. Further, since the sequence $\{\Lambda_{k}\}_{k\ge 1}$ converges to an element $\bar{\Lambda}\in\Lscr_+$, then $\{G^{*}\Lambda_{k}G\}_{k\ge 1}$ converges uniformly to the function $G^{*}(e^{j\theta})\bar{\Lambda}G(e^{j\theta})$ which is positive definite for all $\theta\in[-\pi,\pi]$. Hence, there exists $\mu>0$ such that $\mu_{k}\ge \mu$ for all $k$. On the other hand, since $\delta\Lambda$ is fixed, it must hold that $G^*\delta\Lambda G\leq MI$, where
	\[M:=\max_\theta\rho\left(G^*(e^{j\theta})\delta\Lambda G(e^{j\theta})\right).\]
	Here $\rho(\cdot)$ denotes the spectral radius of a matrix. Therefore, we have
	\[\delta\tilde{F}_{\Lambda_k,\delta\Lambda}\leq M\mu^{-2}GG^*,\ \forall k.\]
    Moreover,
	\[
	\left|[\delta\tilde{F}_{\Lambda_k,\delta\Lambda}]_{i\ell}\right|\le M\mu^{-2} G_{\max}, \quad \forall\, k\ge 1,\ \forall\, i,\ell,
	\]
	where $G_{\max}:=\max_{\theta,i,\ell} \left|[GG^{*}]_{i\ell}\right|<\infty$ since the entries of $G(e^{j\theta})G^{*}(e^{j\theta})$ are continuous functions of $\theta$, analytic in an open annulus containing the unit circle. Hence, by Lebesgue's dominated convergence theorem, we have
    \begin{align*}
	\lim_{k\to\infty}\delta\tilde{\omega}_{\Lambda_k}(\delta\Lambda) = -\int \lim_{k\to\infty} \delta\tilde{F}_{\Lambda_k,\delta\Lambda}=  \delta\tilde{\omega}_{\bar{\Lambda}}(\delta\Lambda).
    \end{align*}
	Partial derivatives can then be recovered by the operation $\langle\delta\Lambda_1,\delta\tilde{\omega}_\Lambda(\delta\Lambda_2)\rangle$  by choosing $\delta\Lambda_k,\ k=1,2$ to be the standard basis matrices of $\Hfrak_n$, where the notation $\langle M_1,M_2\rangle:=\trace(M_1M_2)$ denotes the standard inner product in $\Hfrak_n$. In this way, one can see that every partial derivative of $\tilde{\omega}$ is continuous in $\Lambda$.\end{proof}


\begin{lemma}\label{lem_homotopy}
	The map
	\begin{equation}\label{homotopy}
		\begin{split}
			H\colon \Lscr_+^\Gamma\times[0,1] &\to\range_+\Gamma \\
			(\Lambda,t) & \mapsto\int G\Phi_{\Lambda,t} G^*.
		\end{split}
	\end{equation}
	is a proper continuous homotopy between $\omega$ and $\tilde{\omega}$, where
	\begin{equation}\label{Phi_Gamma_t}
		\Phi_{\Lambda,t}:=W_\Lambda^{-1}[\,t\Psi+(1-t)I\,] W_\Lambda^{-*}.
	\end{equation}
\end{lemma}
\begin{proof}
	By definition we need to show two things, namely that $H$ is jointly continuous in $\Lambda$ and $t$ and that $H$ is proper. 
	In order to prove joint continuity, we first notice that the spectral factor $W_{\Lambda}(z)$ can be written as \cite[Lemma 11.4.1]{FPZ-10}
	\begin{equation}\label{W_spec_factor}
	    W_{\Lambda} (z) := L_{\Lambda}^{-*}B^{*}P_{\Lambda}A(zI -A)^{-1}B + L_{\Lambda},
	\end{equation}
	where $P_{\Lambda}$ is the unique stabilizing solution of the following Discrete-time Algebraic Riccati Equation (DARE)
	\[
	\Pi = A^{*}\Pi A- A^{*}\Pi B(B^{*}\Pi B)^{-1}B^{*}\Pi A+\Lambda,
	\]
	and $L_{\Lambda}$ is the right Cholesky factor of $B^{*}P_{\Lambda}B$, 
	{i.e.,
	\[B^*P_{\Lambda}B=L_{\Lambda}^*L_{\Lambda}\]
	with $L_\Lambda$ being lower triangular having real and positive diagonal entries. Next, let us introduce a change of variables by letting
	\[C_\Lambda:=L_{\Lambda}^{-*}B^{*}P_{\Lambda}.\]
	Then, it is not difficult to recover the relation $L_\Lambda=C_\Lambda B$. In this way, the spectral factor (\ref{W_spec_factor}) can be rewritten as
	\[W_{\Lambda} (z) = C_{\Lambda}A(zI -A)^{-1}B + C_{\Lambda}B.\]
	According to \cite[Thm. A.5.5]{avventi2011spectral}, the dependence of the $m\times n$ matrix $C_\Lambda$ defined above on $\Lambda\in\Lscr_+^\Gamma$ turns out to be a homeomorphism. 
	From this fact it follows that $W_{\Lambda}(e^{j\theta})$ depends continuously on $\Lambda\in\Lscr_+^\Gamma$, for all $\theta\in[-\pi,\pi]$.}
	Consider now
	\[
	\Phi_{\Lambda,t}(e^{j\theta})=W_\Lambda^{-1}(e^{j\theta})[\,t\Psi(e^{j\theta})+(1-t)I\,] W_\Lambda^{-*}(e^{j\theta}).
	\]
	As a linear combination in $t\in[0,1]$ of continuous functions of $\Lambda$, $\Phi_{\Lambda,t}(e^{j\theta})$ is jointly continuous w.r.t. $t\in[0,1]$ and $\Lambda\in\Lscr_+^\Gamma$, for all $\theta\in[-\pi,\pi]$.
	
	Next we need to show the continuity together with the integral. Consider any sequence $\{(\Lambda_{k},t_{k})\}_{k\ge 1}\subset\Lscr_+^\Gamma\times [0,1]$ such that $\lim_{k\to\infty}t_{k}=\bar{t}\in[0,1]$ and $\lim_{k\to\infty}\Lambda_{k}=\bar{\Lambda}\in\Lscr_+^\Gamma$. {Following the same line of reasoning as in the proof of Lemma \ref{lem_smooth_omega_tilde}, there exists $\mu>0$ such that $G^*\Lambda_k G\geq\mu I,\ \forall k$. Therefore, it holds that
	\begin{align*}
		G\Phi_{\Lambda_{k},t_{k}}G^{*}&\le K G(G^*\Lambda_{k} G)^{-1}G^{*}\\
		&\le K \mu^{-1} GG^{*}, \quad \forall\, k\ge 1,
	\end{align*}
	where $K$ is a positive real number such that \[t\Psi(e^{j\theta})+(1-t)I\leq KI,\ \forall\, t\in[0,1],\ \theta\in[-\pi,\pi].\] 
	Such $K$ exists since $\Psi$ is bounded. The rest argument is also similar. Given the joint continuity of $\Phi_{\Lambda,t}$ in $\Lambda$ and $t$, one can show that the following limit holds}
	\begin{align*}
		\lim_{k\to\infty}\int G\Phi_{\Lambda_{k},t_{k}} G^* = \int \lim_{k\to\infty} G\Phi_{\Lambda_{k},t_{k}} G^*=  \int G\Phi_{\bar{\Lambda},\bar{t}} G^*.
	\end{align*}
	This proves joint continuity of $H$ in $t$ and $\Lambda$.

	Once we have joint continuity, the properness is not difficult to prove. In fact, let $K\subset\range_+\Gamma$ be a compact subset, and we next show that the set
	\[H^{-1}(K):=\{\,(\Lambda,t)\in\Lscr_+^\Gamma\times[0,1]\,:\,H(\Lambda,t)\in K\,\}\]
	is compact. The argument is essentially the same as the proof of Theorem 11.4.1 of \cite{FPZ-10}. Since our setting is finite-dimensional, a set being compact is equivalent to being closed and bounded. If $H^{-1}(K)$ is unbounded, one can then find a sequence $\{(\Lambda_k,t_k)\}\subset H^{-1}(K)$ such that $\|(\Lambda_k,t_k)\|\to\infty$ as $k\to\infty$, which necessarily implies $\|\Lambda_k\|\to\infty$. However, in this case $H(\Lambda_k,t_k)$ would tend to be singular, which contradicts the premise of $K$ being compact. This proves the boundedness.
	
	To prove the closedness, if a sequence $\{(\Lambda_k,t_k)\}$ in $H^{-1}(K)$ converges to $(\Lambda,t)$, then $\Lambda$ cannot be on the boundary of $\Lscr_+$, otherwise $\|H(\Lambda_k,t_k)\|\to\infty$, which again contradicts the compactness of $K$. 
	To see the latter fact, notice that
	\begin{align*}
		H(\Lambda_k,t_k) &=\int G\Phi_{\Lambda_{k},t}G^{*} \\
	 &= \int GW_{\Lambda_{k}}^{-1}[\,t\Psi+(1-t)I\,] W_{\Lambda_{k}}^{-*}G^{*} \\ 
	 &\ge K_{\min}\int G(G^*\Lambda_{k} G)^{-1} G^*,
	\end{align*}
	where $K_{\min}:=\min_{t,\theta}\lambda_{\min}\left(t\Psi(e^{j\theta})+(1-t)I\right)> 0$ since $\Psi$ is coercive. Now if $\{\Lambda_{k}\}$ approaches $\partial\Lscr_+$, then $G^*(e^{j\theta})\Lambda_{k} G(e^{j\theta})$ tends to be singular for some $\theta$. Since $G$ has rank $m$ on $\Tbb$, this in turn implies that $\|H(\Lambda_k,t_k)\|\to\infty$ as $k\to\infty$.
	Therefore, by the joint continuity of $H$, $(\Lambda,t)\in H^{-1}(K)$. This concludes the proof of properness.
\end{proof}

\begin{thm}\label{thm_exist}
	The map $\omega$ is surjective.
\end{thm}
\begin{proof}
	Given the second listed property of the degree, the claim follows directly if we can show that
	\[\deg(\omega)\neq0.\]
	We notice first that $\omega$ is proper by Theorem 11.4.1 from \cite{FPZ-10}, and thus the degree is well defined.
	By Lemma \ref{lem_homotopy} and the homotopy invariance of the degree,
	\[\deg(\omega)=\deg(\tilde{\omega}).\]
	As a consequence of Sard--Brown theorem \cite[p.~63]{OR_degree}, the codomain $\range_+\Gamma$ must contain a regular value of $\tilde{\omega}$ since it has positive $\range\,\Gamma$-Lebesgue measure. {By Lemma \ref{lem_smooth_omega_tilde}, the $C^1$ degree (\ref{def_degree}) of $\tilde{\omega}$ at a regular value is well-defined.} Meanwhile, from Theorem 11.4.2 of \cite{FPZ-10}, we know that $\tilde{\omega}$ is bijective. Therefore, we must have
	\[\deg(\tilde{\omega})\neq0,\]
	and this concludes the proof.	
\end{proof}

\subsection{The special case of covariance extension}\label{SubsecCovExt}

Given $\Lambda\in\Lscr_+$ and $G(z)$ in (\ref{trans_func_covext}), $G^*\Lambda G$ is now a matrix Laurent polynomial that takes positive definite values on the unit circle. Let us take
\begin{equation}
	Q(z):=\sum_{k=-p}^{p} Q_kz^{k}\equiv G^*\Lambda G,\quad Q_{-k}=Q_k^*\in\Cbb^{m\times m}.
\end{equation}
{Then according e.g. to \cite{BF16},} 
$Q(z)$ admits a spectral factorization
\begin{equation}
	Q(z)=D^{*}(z)D(z),
\end{equation}
where $D(z)=\sum_{k=0}^{p}D_kz^{-k}$ is a $m\times m$ matrix polynomial (with negative powers) and the scalar polynomial $\det D(z)$ has all its roots strictly inside the unit circle. We shall call such $D(z)$ Schur.\footnote{Moreover, one can make such spectral factor unique if the constant matrix coefficient $D_0$ is required to be lower triangular with real and positive diagonal elements.} Therefore, the outer spectral factor in (\ref{W_Lambda}) is just
\begin{equation}
	W_\Lambda(z)\equiv D(z).
\end{equation}
We have the following corollary of Theorem \ref{thm_exist}.
\begin{cor}
	Given a finite $m\times m$ matrix covariance sequence $\Sigma_0,\Sigma_1,\dots,\Sigma_p$, for any $\Psi\in\Sscr_m$, there exists a Schur polynomial $D(z)$ of degree $p$ such that the spectral density
	\begin{equation}\label{Phi_CovExt}
		\Phi:=D^{-1}\Psi D^{-*}
	\end{equation}
	satisfies the moment equations $(\ref{mmt_eqns})$. The polynomial $D(z)$ is a right Schur spectral factor of $G^*\Lambda G$ for some $\Lambda\in\Lscr_+^\Gamma$.
\end{cor}

In particular, when taking $\Psi(z)=N(z)N^*(z)$ with $N(z)=\sum_{k=0}^{p}N_k\,z^{-k},\ N_k\in\Cbb^{m\times m}$, which is the spectral density of a moving-average process, the spectral density $\Phi$ in (\ref{Phi_CovExt}) would correspond to an $m$-dimensional vector ARMA process
\begin{equation}
\sum_{k=0}^{p}D_k\,y(t-k)=\sum_{k=0}^{p}N_k\,w(t-k),\quad t\in\Zbb,
\end{equation}
and we recover one of the main results of \cite[Section V]{Gthesis} under a more general setting.

\section{Concluding remarks}
We have shown that the multivariate spectral estimation problem posed in \cite{FPZ-10} admits a solution under an arbitrary prior matrix density. An immediate question is uniqueness of the solution and, more strongly, well-posedness of the problem. Following previous work on rational covariance extension \cite{BLGM-95}, we intend to pursue the uniqueness problem in the frame of the global inverse function theorem now attributed to Hadamard. The next theorem appears in \cite{Gordon_diffeo-72}; see also \cite[p.~127]{KP_imp_fcn_thm}. 

\begin{thm}[Hadamard]\label{Hadamard}
Let $M_1$ and $M_2$ be connected, oriented, boundary-less $n$-dimensional manifolds of class $C^1$, and suppose that $M_2$ is simply connected. Then a $C^1$ map $f\colon M_1 \to M_2$ is a diffeomorphism if and only if $f$ is proper and the Jacobian determinant of $f$ never vanishes.
\end{thm}

The next proposition is simple.
\begin{prop}
	The set $\range_+\Gamma$ is simply connected.
\end{prop}
\begin{proof}
	By definition \cite[p.~127]{KP_imp_fcn_thm}, we need to show that: whenever $f\colon[0,1]\to\range_+\Gamma$ is a closed curve, i.e., $f$ is continuous with $f(0)=f(1)=\Sigma$, there exists a continuous function $F\colon [0,1]\times[0,1]\to\range_+\Gamma$ such that
	\begin{enumerate}[{(i)}]
		\item $F(t,0)=f(t)$, for all $t\in[0,1]$,
		\item $F(0,u)=F(1,u)=\Sigma$, for all $u\in[0,1]$, and
		\item $F(t,1)=\Sigma$, for all $t\in[0,1]$.
	\end{enumerate}
	One can easily verify that $F(t,u):=(1-u)f(t)+u\Sigma$ is the desired function.
\end{proof}

Given Theorem \ref{Hadamard} and the above proposition, the open question becomes: Is our map $\omega$ continuously differentiable? If so, how to compute its Jacobian?

{Another research direction concerns the computation of a solution to the problem. To accomplish this task, in \cite{FPZ-10} the following matricial fixed-point iteration was introduced
\begin{align}\label{eq:PFalgmult}
	\Lambda_{k+1}=\int \Lambda_{k}^{1/2} G(W_{\Lambda_{k}}^{-1}\Psi W_{\Lambda_{k}}^{-*})G^{*}\Lambda_{k}^{1/2},
\end{align}
where the initialization is set to $\Lambda_{0}=\frac{1}{n}I$. Iteration \eqref{eq:PFalgmult} can be seen as a multivariate generalization of the scalar algorithm proposed in \cite{PavonF-06} for the Kullback--Leibler estimation of spectral densities. The latter algorithm has proved to be extremely efficient and numerically robust, and its convergence properties have been thoroughly investigated in \cite{FPR-07,FRT-11,B17}. The extension of these convergence results to the multivariate case will be another subject of future investigation.
}

\bibliographystyle{IEEEtran}
\bibliography{Bibliography}

\begin{thebibliography}{10}
\providecommand{\url}[1]{#1}
\csname url@samestyle\endcsname
\providecommand{\newblock}{\relax}
\providecommand{\bibinfo}[2]{#2}
\providecommand{\BIBentrySTDinterwordspacing}{\spaceskip=0pt\relax}
\providecommand{\BIBentryALTinterwordstretchfactor}{4}
\providecommand{\BIBentryALTinterwordspacing}{\spaceskip=\fontdimen2\font plus
\BIBentryALTinterwordstretchfactor\fontdimen3\font minus
  \fontdimen4\font\relax}
\providecommand{\BIBforeignlanguage}[2]{{%
\expandafter\ifx\csname l@#1\endcsname\relax
\typeout{** WARNING: IEEEtran.bst: No hyphenation pattern has been}%
\typeout{** loaded for the language `#1'. Using the pattern for}%
\typeout{** the default language instead.}%
\else
\language=\csname l@#1\endcsname
\fi
#2}}
\providecommand{\BIBdecl}{\relax}
\BIBdecl

\bibitem{FPZ-10}
A.~Ferrante, M.~Pavon, and M.~Zorzi, ``Application of a global inverse function
  theorem of {B}yrnes and {L}indquist to a multivariable moment problem with
  complexity constraint,'' in \emph{Three Decades of Progress in Control
  Sciences}.\hskip 1em plus 0.5em minus 0.4em\relax Springer Berlin Heidelberg,
  2010, pp. 153--167.

\bibitem{BGL-THREE-00}
C.~I. Byrnes, T.~T. Georgiou, and A.~Lindquist, ``A new approach to spectral
  estimation: {A} tunable high-resolution spectral estimator,'' \emph{IEEE
  Transactions on Signal Processing}, vol.~48, no.~11, pp. 3189--3205, 2000.

\bibitem{Georgiou-L-03}
T.~T. Georgiou and A.~Lindquist, ``{K}ullback--{L}eibler approximation of
  spectral density functions,'' \emph{IEEE Transactions on Information Theory},
  vol.~49, no.~11, pp. 2910--2917, 2003.

\bibitem{Kalman}
R.~E. K{\'a}lm{\'a}n, ``Realization of covariance sequences,'' in
  \emph{Toeplitz Centennial}.\hskip 1em plus 0.5em minus 0.4em\relax Springer,
  1982, pp. 331--342.

\bibitem{Gthesis}
T.~T. Georgiou, ``Partial realization of covariance sequences,'' Ph.D.
  dissertation, University of Florida, Gainesville, 1983.

\bibitem{Georgiou-87}
------, ``Realization of power spectra from partial covariance sequences,''
  \emph{IEEE Transactions on Acoustics, Speech, and Signal Processing},
  vol.~35, no.~4, pp. 438--449, 1987.

\bibitem{BLGM-95}
C.~I. Byrnes, A.~Lindquist, S.~V. Gusev, and A.~S. Matveev, ``A complete
  parameterization of all positive rational extensions of a covariance
  sequence,'' \emph{IEEE Transactions on Automatic Control}, vol.~40, no.~11,
  pp. 1841--1857, 1995.

\bibitem{BGL-98}
C.~I. Byrnes, S.~V. Gusev, and A.~Lindquist, ``A convex optimization approach
  to the rational covariance extension problem,'' \emph{SIAM Journal on Control
  and Optimization}, vol.~37, no.~1, pp. 211--229, 1998.

\bibitem{SIGEST-01}
------, ``From finite covariance windows to modeling filters: {A} convex
  optimization approach,'' \emph{SIAM Review}, vol.~43, no.~4, pp. 645--675,
  2001.

\bibitem{Georgiou-87-NP}
T.~T. Georgiou, ``A topological approach to {N}evanlinna--{P}ick
  interpolation,'' \emph{SIAM Journal on Mathematical Analysis}, vol.~18,
  no.~5, pp. 1248--1260, 1987.

\bibitem{BGL-01}
C.~I. Byrnes, T.~T. Georgiou, and A.~Lindquist, ``A generalized entropy
  criterion for {N}evanlinna--{P}ick interpolation with degree constraint,''
  \emph{IEEE Transactions on Automatic Control}, vol.~46, no.~6, pp. 822--839,
  2001.

\bibitem{BLN-03}
A.~Blomqvist, A.~Lindquist, and R.~Nagamune, ``Matrix-valued
  {N}evanlinna--{P}ick interpolation with complexity constraint: An
  optimization approach,'' \emph{IEEE Transactions on Automatic Control},
  vol.~48, no.~12, pp. 2172--2190, 2003.

\bibitem{KreinNudelman}
M.~G. Kre\u{\i}n and A.~A. Nudel'man, \emph{The {M}arkov moment problem and
  extremal problems}, ser. Translations of Mathematical Monographs.\hskip 1em
  plus 0.5em minus 0.4em\relax Providence, Rhode Island: American Mathematical
  Society, 1977, vol.~50.

\bibitem{Grenander_Szego}
U.~Grenander and G.~Szeg{\"o}, \emph{Toeplitz forms and their applications},
  ser. California Monographs in Mathematical Sciences.\hskip 1em plus 0.5em
  minus 0.4em\relax University of California Press, 1958.

\bibitem{PavonF-06}
M.~Pavon and A.~Ferrante, ``On the {G}eorgiou--{L}indquist approach to
  constrained {K}ullback--{L}eibler approximation of spectral densities,''
  \emph{IEEE Transactions on Automatic Control}, vol.~51, no.~4, pp. 639--644,
  2006.

\bibitem{FPR-07}
A.~Ferrante, M.~Pavon, and F.~Ramponi, ``Further results on the
  {B}yrnes--{G}eorgiou--{L}indquist generalized moment problem,'' in
  \emph{Modeling, Estimation and Control}.\hskip 1em plus 0.5em minus
  0.4em\relax Springer Berlin Heidelberg, 2007, pp. 73--83.

\bibitem{FRT-11}
A.~Ferrante, F.~Ramponi, and F.~Ticozzi, ``On the convergence of an efficient
  algorithm for {K}ullback--{L}eibler approximation of spectral densities,''
  \emph{IEEE Transactions on Automatic Control}, vol.~56, no.~3, pp. 506--515,
  2011.

\bibitem{B17}
G.~Baggio, ``Further results on the convergence of the {P}avon--{F}errante
  algorithm for spectral estimation,'' \emph{IEEE Trans. Autom. Control, (to
  appear)}, 2018.

\bibitem{Z14rat}
M.~Zorzi, ``Rational approximations of spectral densities based on the alpha
  divergence,'' \emph{Mathematics of Control, Signals, and Systems}, vol.~26,
  no.~2, pp. 259--278, 2014.

\bibitem{Georgiou-06}
T.~T. Georgiou, ``Relative entropy and the multivariable multidimensional
  moment problem,'' \emph{IEEE Transactions on Information Theory}, vol.~52,
  no.~3, pp. 1052--1066, 2006.

\bibitem{FPR-08}
A.~Ferrante, M.~Pavon, and F.~Ramponi, ``Hellinger versus {K}ullback--{L}eibler
  multivariable spectrum approximation,'' \emph{IEEE Transactions on Automatic
  Control}, vol.~53, no.~4, pp. 954--967, 2008.

\bibitem{RFP-09}
F.~Ramponi, A.~Ferrante, and M.~Pavon, ``A globally convergent matricial
  algorithm for multivariate spectral estimation,'' \emph{IEEE Transactions on
  Automatic Control}, vol.~54, no.~10, pp. 2376--2388, 2009.

\bibitem{FMP-12}
A.~Ferrante, C.~Masiero, and M.~Pavon, ``Time and spectral domain relative
  entropy: A new approach to multivariate spectral estimation,'' \emph{IEEE
  Transactions on Automatic Control}, vol.~57, no.~10, pp. 2561--2575, 2012.

\bibitem{GL-17}
T.~T. Georgiou and A.~Lindquist, ``Likelihood analysis of power spectra and
  generalized moment problems,'' \emph{IEEE Transactions on Automatic Control},
  vol.~62, no.~9, pp. 4580--4592, 2017.

\bibitem{Z14}
M.~Zorzi, ``A new family of high-resolution multivariate spectral estimators,''
  \emph{IEEE Transactions on Automatic Control}, vol.~59, no.~4, pp. 892--904,
  2014.

\bibitem{Z15}
------, ``Multivariate spectral estimation based on the concept of optimal
  prediction,'' \emph{IEEE Transactions on Automatic Control}, vol.~60, no.~6,
  pp. 1647--1652, 2015.

\bibitem{PZ-converge-17}
B.~Zhu and G.~Picci, ``Proof of local convergence of a new algorithm for
  covariance matching of periodic {ARMA} models,'' \emph{IEEE Control Systems
  Letters}, vol.~1, no.~1, pp. 206--211, 2017.

\bibitem{Picci-Z-017}
G.~Picci and B.~Zhu, ``Approximation of vector processes by covariance matching
  with applications to smoothing,'' \emph{IEEE Control Systems Letters},
  vol.~1, no.~1, pp. 200--205, 2017.

\bibitem{BL_global_inverse-07}
C.~I. Byrnes and A.~Lindquist, ``Interior point solutions of variational
  problems and global inverse function theorems,'' \emph{International Journal
  of Robust and Nonlinear Control}, vol.~17, no. 5-6, pp. 463--481, 2007.

\bibitem{Zorzi-F-12}
M.~Zorzi and A.~Ferrante, ``On the estimation of structured covariance
  matrices,'' \emph{Automatica}, vol.~48, no.~9, pp. 2145--2151, 2012.

\bibitem{FPZ-12}
A.~Ferrante, M.~Pavon, and M.~Zorzi, ``A maximum entropy enhancement for a
  family of high-resolution spectral estimators,'' \emph{IEEE Transactions on
  Automatic Control}, vol.~57, no.~2, pp. 318--329, 2012.

\bibitem{ning2013geometry}
L.~Ning, X.~Jiang, and T.~Georgiou, ``On the geometry of covariance matrices,''
  \emph{IEEE Signal Processing Letters}, vol.~20, no.~8, pp. 787--790, 2013.

\bibitem{Georgiou-02}
T.~T. Georgiou, ``The structure of state covariances and its relation to the
  power spectrum of the input,'' \emph{IEEE Transactions on Automatic Control},
  vol.~47, no.~7, pp. 1056--1066, 2002.

\bibitem{georgiou2002spectral}
------, ``Spectral analysis based on the state covariance: the maximum entropy
  spectrum and linear fractional parametrization,'' \emph{IEEE transactions on
  Automatic Control}, vol.~47, no.~11, pp. 1811--1823, 2002.

\bibitem{RFP-10-well-posedness}
F.~Ramponi, A.~Ferrante, and M.~Pavon, ``On the well-posedness of multivariate
  spectrum approximation and convergence of high-resolution spectral
  estimators,'' \emph{Systems \& Control Letters}, vol.~59, no.~3, pp.
  167--172, 2010.

\bibitem{OR_degree}
E.~Outerelo and J.~M. Ruiz, \emph{Mapping degree theory}.\hskip 1em plus 0.5em
  minus 0.4em\relax Providence, Rhode Island: American Mathematical Society,
  2009, vol. 108.

\bibitem{JTSchwartz}
J.~T. Schwartz, \emph{Nonlinear functional analysis}.\hskip 1em plus 0.5em
  minus 0.4em\relax CRC Press, 1969.

\bibitem{brookes2005matrix}
M.~Brookes, ``The matrix reference manual,'' \emph{Imperial College London},
  2005.

\bibitem{Bhatia}
R.~Bhatia, \emph{Matrix analysis}, ser. Graduate Texts in Mathematics.\hskip
  1em plus 0.5em minus 0.4em\relax Graduate Texts in Mathematics,
  Springer-Verlag New York, 2013, vol. 169.

\bibitem{avventi2011spectral}
E.~Avventi, ``Spectral moment problems: generalizations, implementation and
  tuning,'' Ph.D. dissertation, KTH Royal Institute of Technology, Stockholm,
  2011.

\bibitem{BF16}
G.~Baggio and A.~Ferrante, ``On the factorization of rational discrete-time
  spectral densities,'' \emph{IEEE Transactions on Automatic Control}, vol.~61,
  no.~4, pp. 969--981, 2016.

\bibitem{Gordon_diffeo-72}
W.~B. Gordon, ``On the diffeomorphisms of {E}uclidean space,'' \emph{The
  American Mathematical Monthly}, vol.~79, no.~7, pp. 755--759, 1972.

\bibitem{KP_imp_fcn_thm}
S.~G. Krantz and H.~R. Parks, \emph{The implicit function theorem: {H}istory,
  theory, and applications}, ser. Modern Birkh\"{a}user Classics.\hskip 1em
  plus 0.5em minus 0.4em\relax Birkh\"{a}user Basel, 2013.

\end{thebibliography}

\end{document}